\documentclass{amsart}
\usepackage{amssymb,amscd}

 \newtheorem{thm}{Theorem}[section]
 \newtheorem{cor}[thm]{Corollary}
 \newtheorem{lemma}[thm]{Lemma}
 \newtheorem{prop}[thm]{Proposition}
 \theoremstyle{definition}
 \newtheorem{defn}[thm]{Definition}
 \newtheorem{exmp}[thm]{Example}
 \theoremstyle{remark}
 \newtheorem{rem}[thm]{Remark}
 
 \newtheorem{quest}[thm]{Question}
 
 \numberwithin{equation}{subsection}
 \newtheorem{ack}{Acknowledgment}
 
\newcommand{\TT}{\text{$\mathcal{T}$}}
\newcommand{\BB}{\text{$\mathcal{B}$}}

\newcommand{\UU}{\text{$\mathcal{U}$}}
\newcommand{\FF}{\text{$\mathcal{F}$}}
\newcommand{\GG}{\text{$\mathcal{G}$}}

\newcommand{\LB}{\text{$\Lambda$}}
\newcommand{\sg}{\text{$\sigma$}}

\newcommand{\intr}{\operatorname{int}}

\newcommand{\id}{\operatorname{id}}

\newcommand{\Ker}{\operatorname{Ker}}
\newcommand{\Image}{\operatorname{Im}}
\newcommand{\Star}{\operatorname{Star}}
\newcommand{\MT}{\operatorname{MT}}
\newcommand{\Meas}{\operatorname{Meas}}
\newcommand{\length}{\operatorname{length}}

        \newcommand{\field}[1]{\text{$\mathbb{#1}$}}
        \newcommand{\N}{\field{N}}
        \newcommand{\Z}{\field{Z}}
        \newcommand{\Q}{\field{Q}}
        \newcommand{\R}{\field{R}}
        \newcommand{\C}{\field{C}}

\begin{document}

\title{Cohomology of measurable laminations}

\author{Carlos Meni\~no Cot\'on}

\address{Departamento de Xeometr\'{\i}a e Topolox\'{\i}a\\
         Facultade de Matem\'aticas\\
         Universidade de Santiago de Compostela\\
         15782 Santiago de Compostela}

\email{carlos.meninho@gmail.com}

\thanks{Supported by MICINN (Spain): FPU program and Grant MTM2008-02640}


\begin{abstract}
A new notion of cohomology is introduced for MT-spaces, which are
measurable and topological spaces whose measurable
structure may not agree with the Borel \sg-algebra of their
topology. The main examples of MT-spaces are measurable foliations. This is a singular version of the measurable simplicial
cohomology defined by Heitsch and Lazarov for foliations \cite{Heitsch-Lazarov} and extended by Bermudez for MT-spaces \cite{Bermudez}. Basic
topics of algebraic topology are adapted, and applications to
the theory of foliations are given. Moreover we introduce a new
notion of singular $L^2$-cohomology for MT-spaces.
\end{abstract}

\maketitle

\section*{Introduction}
The theory of MT-spaces is mainly devoted to the study of foliations. A usual foliation induces an MT-space by choosing the Borel \sg-algebra of the ambient space of the foliation and the leaf topology (whose connected components are the leaves). The resulting MT-space may be thought of as a simplification of the foliation where the transverse topology is dropped off and only the leaf topology and the global measurable structure remain. Thus the transverse topological dynamics does not make sense. However the transverse measurable dynamics is now relevant, and MT-spaces are the appropriate abstract setting to study it. For instance, a notion of holonomy transformation can be defined, but the holonomy germs cannot be considered. In general, transversely measurable invariants of foliations (or invariants of manifolds) can be defined in the MT-setting, where they may become easier to deal with.

The first section introduces the basic theory and notation of MT-spaces, used in the rest of the paper, following the work of M.~Berm\'udez and G.~Hector \cite{Bermudez-Hector}.

The second section introduces the concept of measurable singular cohomology. Its definition, possible variants, and corresponding versions of basic results of algebraic topology are given, like excision, Mayer-Vietoris or its homotopy invariance.

The third section deals with other versions of measurable cohomology, like the simplicial or $L^2$ versions, which can be found in the works of M.~Bermudez \cite{Bermudez} and J.L.~Heitsch and C.~Lazarov~\cite{Heitsch-Lazarov}. We prove that our singular version is isomorphic to the simplicial one.

The fourth section is devoted to a new concept of measurable singular $L^2$-cohomology. We prove that it is a homotopy invariant under some conditions, and that it is isomorphic to the $L^2$-simplicial cohomology defined in \cite{Heitsch-Lazarov}. We think that this new cohomology could be used to give a new proof of the homotopy invariance of the $L^2$-Betti numbers proved by J.L.~Heitsch and C.~Lazarov \cite{Heitsch-Lazarov}.

 A final section of examples is given. For instance, the measurable cohomology of degree $n$ is computed for the Kr\"onecker type foliation of the $n+1$ torus by minimal hyperplanes of dimension $n$, showing that it is not trivial.

\section{MT-spaces and measurable laminations}

A {\em measurable topological space\/}, or {\em MT-space\/}, is
a set $X$ equipped with a \sg-algebra and a topology. Usually, measure
theoretic concepts will refer to the \sg-algebra of $X$, and
topological concepts will refer to its topology; in
general, the \sg-algebra is different from the Borel \sg-algebra
induced by the topology. An {\em MT-map\/} between MT-spaces is a
measurable continuous map. An {\em MT-isomorphism} is a map between MT-spaces that is a measurable isomorphism and a homeomorphism.

Trivial examples are topological spaces with de Borel
$\sg$-algebra, and measurable spaces with the discrete topology. Let
$X$ and $Y$ be MT-spaces. Suppose that there exists a measurable
embedding $i:X\to Y$ that maps measurable sets to measurable sets. Then $X$ is called an {\em MT-subspace\/} of $Y$. Notice that if $X$
and $Y$ are standard, the measurability of $i$ means that it maps Borel sets to Borel sets
\cite{Srivastava}. The product $X\times Y$ is an MT-space too with the product topology and the $\sg$-algebra generated by
products of measurable sets of $X$ and $Y$.

Let $R$ be an equivalence relation on an MT-space $X$. In order to
give an MT-structure to the quotient $X/R$, consider the quotient topology and the $\sg$-algebra
generated by the projections of measurable saturated sets of $X$.

A {\em Polish space\/} is a completely metrizable and separable
topological space. A {\em standard Borel space\/} is a measurable
space isomorphic to a Borel subset of a Polish space. Let $T$ be a
standard Borel space and let $P$ be a Polish space. $P\times T$
will be endowed with the structure of MT-space defined by the
\sg-algebra generated by products of Borel subsets of $T$ and
Borel subsets of $P$, and the product of the discrete topology on
$T$ and the topology of $P$.

A {\em measurable chart} on an MT-space $X$ is an
MT-isomorphism $\varphi:U\to P\times T$, where $U$ is open and
measurable in $X$, $T$ is a standard Borel space, and $P$ is locally compact, connected and
locally path connected Polish space; let us remark that $P$ and $T$ depend on the chart. The sets $\varphi^{-1}(P\times\{\ast\})$ are
called {\em plaques\/} of $\varphi$, and the sets
$\varphi^{-1}(\{\ast\}\times T)$ are called {\em transversals\/}
associated to $\varphi$. A {\em measurable atlas\/} on $X$
is a countable family of measurable charts whose domains
cover $X$. A {\em measurable lamination\/} is an MT-space
that admits a measurable atlas. Observe that we
always consider countable atlases, therefore the ambient space is
also a standard space. The connected components of $X$ are called
its {\em leaves\/}. An example of measurable lamination is
a usual foliation with its Borel \sg-algebra and the leaf
topology. According to this definition, the leaves are second
countable connected manifolds, but they may not be Hausdorff. If $P\simeq \R^m$, it
is possible to define a concept of $C^r$ tangential structure; in this setting, it cannot be defined as a maximal atlas with (tangentially) $C^r$ changes of coordinates because the atlases are required to be countable, but we proceed as follows. A measurable atlas is said to be (tangentially) $C^r$ if its coordinate changes are (tangentially) $C^r$. Then a $C^r$ structure is an equivalence class of $C^r$ measurable atlases, where two $C^r$ measurable atlases are equivalent if their union is a $C^r$ measurable atlas.

The term ``lamination'' (or ``measurable lamination'') is commonly used when the leaves are manifolds. Thus the term ``measurable Polish lamination'' could be more appropiate in our setting. But we simply write ``measurable lamination'' for the sake of simplicity.

A measurable subset $T\subset X$  is called a {\em transversal\/}
if its intersection with each leaf is countable \cite{Heitsch-Lazarov}; these are
slightly more general than the transversals of \cite{Bermudez}. Let
$\TT(X)$ be the family of transversals of $X$. This set is closed
under countable unions and intersections, but it is not a
\sg-algebra. A transversal meeting all leaves is called
{\em complete\/}.

A {\em measurable holonomy transformation\/} is a measurable
isomorphism $\gamma:T\to T'$, for $T, T' \in \TT(X)$, which maps
each point to a point in the same leaf. A {\em transverse
invariant measure\/} on $X$ is a \sg-additive map, $\LB:\TT(X)\to
[0,\infty]$, invariant by measurable holonomy transformations. The
classical definition of transverse invariant measure in the
context of foliated spaces is a measure on topological
transversals invariant by holonomy transformations (see {\em
e.g.\/} \cite{Candel-Conlon}). These two notions of transverse invariant measures agree
for foliated spaces \cite{Connes}.

Our principal tools in this setting are the following two results.

\begin{prop}[Lusin, see {\em e.g.\/} \cite{Srivastava}]\label{p:Lusin1}
Let $X$ and $Y$ be standard Borel spaces and $f:X\to Y$ a
measurable map with countable fibers. Then $f(X)$ is Borel in $Y$
and there exists a measurable section $s:f(X)\to X$ of $f$. In
particular, if $f$ is injective, then $s$ is a Borel isomorphism.
Moreover there exists a countable Borel partition, $X=\bigcup_i
X_i$, so that each restriction $f|_{X_i}$ is injective.
\end{prop}

\begin{thm}[Kunugui, Novikov, see {\em e.g.\/} \cite{Srivastava}]\label{t:Srivastava}
Let $\{V_n\}_{n\in\N}$ be a countable base for a Polish space $P$.
Let $B\subset P\times T$ be a Borel set such that
$B\cap(P\times\{t\})$ is open for every $t\in T$. Then there
exists a sequence $\{B_n\}_{n\in\N}$ of Borel sets of $T$ such
that
$$B=\bigcup_n (V_n\times B_n)\;.$$
\end{thm}


\begin{lemma}\label{l:cocycle}
Let $\varphi_i:U_i\to P\times T_i$ and $\varphi_j:U_j\to P\times
T_j$ be measurable Polish charts of $X$. There exists a sequence
of Borel sets of $T_i$, $\{B_n\}_{n\in\N}$, and a base of $P$,
$\{V_n\}_{n\in\N}$, such that $\varphi_i(U_i\cap
U_j)=\bigcup_n(V_n\times B_n)$ and
$\varphi_j\circ\varphi_i^{-1}(x,t)= (g_{ijn}(x,t), f_{ijn}(t))$
for $(x, t)\in V_n\times B_n$, where each $f_{ijn}$ is a Borel
isomorphism and each $g_{ijn}$ is an MT-map.
\end{lemma}

\begin{proof}
We apply Theorem \ref{t:Srivastava} to $\varphi_j(U_i\cap U_j)$,
with a base $\{V_n\}_{n\in\N}$ consisting of connected open sets of
$P$, and obtain a family of sets $V_n\times B'_n$ such that
$\varphi_j(U_i\cap U_j)=\bigcup_n(V_n\times B'_n)$. Now we apply
Theorem \ref{t:Srivastava} to each set
$\varphi_i\circ\varphi_j^{-1}(V_k\times B'_k)$, $k\in\N$. We
obtain sequences $B'_{k,n}$ such that
$$
\varphi_i\circ\varphi_j^{-1}(V_k\times B'_k)= \bigcup_n(V_n\times
B'_{k,n})\;,\qquad k\in\N\;.
$$
The sets $\varphi_j\circ\varphi_i^{-1}(V_n\times\{t\})$, $t\in
B'_{k,n}$, are contained in a single plaque of the form
$\R^n\times \{\ast\}$ since each $V_n$ is connected. Hence
$$
\varphi_j\circ\varphi_i^{-1}(x,t)= (g_{ijkn}(x,t), f_{ijkn}(t))
$$
for $(x,t)\in V_n\times B'_{k,n}$. We shall show that $f_{ijkn}$
is bijective. If there exist $t,t'\in T_i$ with
$f_{ijkn}(t)=f_{ijkn}(t')=t''$, then $g_{ijkn}(V_n\times\{t\})$
and $g_{ijkn}(V_n\times\{t'\})$ are contained in the plaque
$V_{k}\times\{t''\}$ of $V_k\times B'_{k}$, but this plaque is
image by $\varphi_j\circ\varphi_i^{-1}$ of a given connected open
set since $\varphi_j$ and $\varphi_i$ are homeomorphisms, and
$V_k$ is connected. Thus the image by
$\varphi_i\circ\varphi^{-1}_j$ of the plaque $V_k\times\{t''\}$ is
contained in a single plaque of $U_i$. This contradicts $t\neq t'$.

It is easy to show that the maps $f_{ijkn}$ are measurable since
they can be given as a composition of a projection with the
restriction of the cocycle map $\varphi_j\circ\varphi_i^{-1}$ to
$B'_{k,n}$. Finally, the maps $f_{ijkn}$ are Borel isomorphisms to
their images by Proposition~\ref{p:Lusin1}.
\end{proof}

\begin{defn}
A foliated measurable atlas $\UU$ is called {\em regular\/} if,
for each chart $(U,\varphi)\in\UU$, there exists another
measurable foliated chart $(W,\psi)$ such that the closure of each
plaque in $U$ is compact, $\overline{U}\subset W$,
$\varphi=\psi|_U$, and, for every pair of charts $(U_1,\varphi_1),(U_2,\varphi_2)\in\mathcal{U}$,
each plaque of $(U_1,\varphi_1)$ meets at most one plaque of $(U_2,\varphi_2)$. The set
$\overline{U}$ is also measurable for all $(U,\varphi)\in\UU$.
\end{defn}

This definition is weaker than the corresponding one for usual foliations (see
{\em e.g.\/} \cite{Candel-Conlon}). The locally finite condition
cannot be considered in the measurable setting since there is no
ambient topology. The following result follows from
Lemma~\ref{l:cocycle}.

\begin{cor}
A measurable foliated space with a foliated measurable atlas, such
that each chart meets a finite number of charts, admits a regular
measurable foliated atlas.
\end{cor}

From now on, we always consider measurable laminations that
admit regular measurable foliated atlases.

\begin{exmp}[Measurable wedge]
Let $\FF$ and $\GG$ be measurable laminations, and let $T$
and $T'$ be complete transversals consisting of isolated points in leaves. Suppose that there
exists a measurable bijection $\gamma:T\to T'$ such that
$\gamma:T/\FF\to T'/\GG$ is bijective (the bijection take points
of each leaf to a single leaf in both directions). The {\em
measurable wedge\/} $\FF\vee^T_{\gamma}\GG$, relative to the pair
$(T,\gamma)$, is the quotient MT-space of $\FF\sqcup\GG$ by the
relation $t\sim \gamma(t)$ for $t\in T$. The measurable wedge is a measurable
lamination since $T$ and $T'=\gamma(T)$ consist of
isolated points on the leaves. If $T$ (and $T'$) meets each leaf
in only one point, then the measurable wedge is a measurable
lamination where any leaf is a wedge of two leaves by the latest
condition on $\gamma$. Unfortunately, in many measurable
laminations there is no measurable transversal with
this property.
\end{exmp}

\begin{exmp}[Measurable suspensions]
Let $P$ be a connected, locally path connected and semi-locally $1$-connected Polish space, and let $S$
be a standard space. Let $\Meas (S)$ denote the group of
measurable transformations of $S$. Let
$$h : \pi_{1}(P, x_{0})\to \Meas(S)$$ be a homomorphism. Let
$\widetilde{P}$ the universal covering of $P$ and consider
the action of $\pi_{1}(P, x_{0})$ on the MT-space
$\widetilde{P}\times S$ given by
\begin{align*}
g\cdot(x, y) = (xg^{-1}, h(g)(y))\;.
\end{align*}
The corresponding quotient MT-space, $\widetilde{P}\times_h S$, will be called the {\em measurable suspension\/} of
$h$. $\widetilde{P}\times_h S$ is a measurable lamination, $\{\ast\}\times S$
is a complete transversal, and its leaves are covering spaces of $P$.
\end{exmp}

\begin{exmp}[Measurable graphs]
Measurable graphs are measurable laminations such that
every leaf is a graph in the classical sense. In this setting, any
plaque is a finite wedge of open intervals. Of course, the
measurable wedge of measurable graphs is a measurable graph.
\end{exmp}

\section{Measurable singular cohomology}

The idea of introducing a measurable cohomology for MT-spaces is an attempt to
define an algebraic invariant giving information about the
mixing of their topology and measurable structure. The natural
setting is, surprisingly, the cohomology and not the homology:
measurability has an obvious simple interpretation for singular cochains, whilst it seems to be difficult to introduce for
singular chains. We also suppose that $\Gamma$ is a standard commutative group
or unitary ring; {\em i.e.\/}, $\Gamma$ is an abelian group or unitary ring and a standard space where all the
operations and the inverse map are measurable.

\begin{defn}[Measurable prism]
A {\em meaurable prism\/} is a product of a standard Borel space
$T$ and a linear region of $\R^N$ (for instance a polygon) with
the standard MT-structure. A {\em measurable simplex\/} is a
measurable prism where the topological fiber is a canonical
$n$-simplex $\bigtriangleup^n$. A {\em measurable singular
simplex\/} on $X$ is an MT-map $\sg:\bigtriangleup^n\times T \to
X$.
\end{defn}

Let  $\omega$ be a usual {\em singular $n$-cochain\/} over a
coefficient ring $\Gamma$. It is said that $\omega$ is {\em
measurable\/} if $\omega_{\sg}:T\to\Gamma$, $t\mapsto
\omega(\sg_{|\bigtriangleup\times \{t\}})$, is measurable for all
measurable singular $n$-simplex $\sg$. The set of measurable
cochains is a subcomplex of the complex of usual cochains since
the coboundary operator $\delta$ preserves the measurability. This
measurable subcomplex is denoted by $C^*_{MT}(X,\Gamma)$, and the
coboundary operator restricted to this complex is also denoted by
$\delta$.

The {\em singular measurable cohomology\/} is defined as
usual by
$$H^n_{\MT}(X,\Gamma)=\Ker \delta_n/\Image\delta_{n-1}\;.$$

The usual cup product gives a well defined exterior product on
measurable cochains since the operations in $\Gamma$ are measurable.

The usual formula
$\delta(\omega\smile\theta)=\delta\omega\,\smile\theta
+ (-1)^n\,\omega\smile\delta\theta$ holds. Therefore a cup product is induced in measurable cohomology, obtaining the
graded ring $(H^*_{\MT}(X,\Gamma), +, \smile)$.

Any MT-map $f:X\to Y$ defines a cochain map $f^*:C^*_{\MT}(Y,\Gamma)\to C^*_{\MT}(X,\Gamma)$ by $f^*(\omega)(\sg)=\omega(f\circ\sg)$, which in turn
induces a homomorphism between measurable cohomology groups, $f^*:H_{\text{\rm MT}}^*(Y,\Gamma)\to H_{\text{\rm MT}}^*(X,\Gamma)$.

Let $U\subset X$ be an MT-subspace of $X$. The inclusion map
determines a chain map $i^*:C^*_{\MT}(X,\Gamma)\to
C^*_{\MT}(U,\Gamma)$. The cochain complex $\Ker(i^*)$
will be denoted by $C^*_{\MT}(X,U,\Gamma)$; it consists of the cochains that vanish on any singular
simplex contained in $U$. The corresponding cohomology groups will be called the {\em measurable relative
cohomology groups\/} of $(X,U)$. By using the Ker-Coker Lemma like in the classical case,
there exists a long exact sequence of cohomology groups (the details are easy to check):
$$
\cdots\to H^{n}_{\MT}(X,U,\Gamma)\to H^n_{\MT}(X,\Gamma)\to
H^n_{\MT}(U,\Gamma)\to H^{n+1}_{\MT}(X,U,\Gamma)\to\cdots
$$

\begin{defn}
Let $X,Y$ be MT-spaces. A {\em measurable homotopy\/} or MT-homotopy is an MT-map
$H:X\times [0,1]\to Y$. It is said that $H(-,0)$ and $H(-,1)$ are
MT-{\em homotopic\/} maps.
\end{defn}
\begin{prop}[Invariance by MT-homotopy]\label{p:invariancetransversals}
Let $f,g:X\to Y$ be MT-homotopic maps. Then $f^*=g^*:H_{\text{\rm MT}}^*(Y,\Gamma)\to H_{\text{\rm MT}}^*(X,\Gamma)$.
\end{prop}
\begin{proof}
The proof is a trivial consequence of the classical proof for
singular cohomology. The measurable homotopy induces a chain
homotopy between $f_*$ and $g_*$ at the chain complex level. The definition is given by cutting on the space
$\bigtriangleup^n\times [0,1]$ into a finite number of
$(n+1)$-simplices $\bigtriangleup^{n+1}$, we will denote these simplices
by $\Pi_i$. Let $H$ be the measurable homotopy between $f$ and
$g$. The map $P:C^{n+1}(X,\Gamma)\to C^{n}(X,\Gamma)$ defined by
$P(\omega)(\sg)=\sum_i I_i\,\omega(H\circ\sg_{|\Pi_i})$, where $I_i$
is the orientation factor: $I_i=1$ if the orientation of $\Pi_i$ agrees with the orientation induced by $\bigtriangleup^n\times[0,1]$, and $I_i=-1$ otherwise; here, the orientation in $\bigtriangleup^n\times[0,1]$ is chosen so that it induces the usual orientation of a simplex in the fiber $\bigtriangleup^n\times\{0\}$ (see \cite{Hatcher} for details). This map is a cochain homotopy between $f^*$ and
$g^*$ preserving measurability. Hence $f^*$ and $g^*$ induce
the same homomorphism in measurable cohomology.
\end{proof}

Let $\UU$ be a countable measurable open covering of $X$ and let
$C^\UU_*(X,\Gamma)$ the set of chains whose simplices are contained in elements of $\UU$.
Clearly, $C^\UU_*(X,\Gamma)$ is a chain subcomplex of $C_*(X,\Gamma)$. The cochain complex $C_{\mathcal{U}}^*(X,\Gamma)$, dual of $C^{\mathcal{U}}_*(X,\Gamma)$, will be
called the cochain complex of $X$ associated to the covering
$\UU$. For the measurable setting, it is enough to check the measurability condition
for this type of cochains on simplices $\sigma:\triangle\times T\to X$, with
$\sg(\bigtriangleup\times T)\subset U$ for some $U\in\UU$. The
respective cohomology groups are denoted by
$H^*_{\UU,\MT}(X,\Gamma)$. Our goal now is to show that these
groups are isomorphic to the original ones. The restriction map
induces a natural homomorphism $i:H^n_{\MT}(X,\Gamma)\to H^n_{\UU,\MT}(X,\Gamma)$. To prove that $i$ is an isomorphism, we adapt the classical proof by using nice subdivisions  of singular simplices adapted to
\UU\ in order to give an inverse map preserving measurability.

\begin{defn}
A {\em linear region\/} of $\R^n$ is a compact and connected set defined by a finite union of finite intersections of {\em half-hyperplanes\/}. Here, a {\em half-hyperplane\/} is the set of points $x\in\mathbb{R}^n$ satisfying a linear inequality  $\phi(x)\leq c$, where $\phi:\mathbb{R}^n\to\mathbb{R}$ is a non-zero linear map and $c\in\R$. Any linear region has the structure of manifold with corners. A linear region $R$ is {\em maximal\/} if its dimension is $n$. Given a linear region $R$ of $\R^n$, its boundary
$\partial R$ can be expressed as a union of linear regions on
 affine submanifolds of $\R^n$ of dimension $i<n$; they are called the {\em $i$-faces\/} of the linear region. Two
linear regions of $\R^n$ are said to be {\em attached\/} if their
intersection is a union of faces.
\end{defn}

\begin{lemma}\label{l:attachedunion}
A finite union of maximal linear regions of $\R^n$, $R_1\cup...\cup R_N$, can be
expressed as a finite union of maximal linear regions $R'_i$,
$i\in\{1,\dots,N'\}$, such that $R'_i$ and $R'_j$ are attached or
disjoint for $i\neq j$, each $R'_i$ is contained in
some $R_j$, and $\bigcup_{i=1}^N\partial
R_i=\bigcup_{j=1}^{N'}\partial R'_j$.
\end{lemma}
\begin{proof}
Let $R_1$ and $R_2$ be two maximal linear regions. The connected
components of $R_1\cap R_2$, $\overline{R_1\setminus R_2}$ and
$\overline{R_2\setminus R_1}$ are linear regions. The maximal linear regions of this kind satisfy the
required conditions. Hence the statement is easily proved by
induction on $N$.
\end{proof}

\begin{lemma}\label{l:triangulationlinearregions}
Any finite union of linear regions, such that any pair of them are attached or disjoint, admits a triangulation
$\TT$ that induces a triangulation on each linear region.
\end{lemma}
\begin{proof}
Any linear region can be subdivided by a union of convex linear regions. This subdivision is given by its definition from half-hyperplanes: let $R=\bigcup_{i=1}^K\bigcap_{j=1}^{M_i} H_j^{i,0}$, where the sets $H_j^{i,0}$ are half-hyperplanes, let $H_j^{i,1}=\overline{\R^n\setminus H_j^{i,0}}$. For each $i$, consider the family
$$
D_i=\left\{\,A_1\cap\cdots\cap A_{i-1}\cap\bigcap_{j=1}^{M_i} H_j^{i,0} \cap A_{i+1}\cap\cdots\cap A_K\,\right\}\,,
$$
where $A_k\in\{\,\bigcap_{j=1}^{M_k}\ H_j^{k,i_j}|\ i_j\in\{0,1\}\;\forall\, j\,\}$. The union of the families $D_i$ gives a convex subdivision of $R$. Hence we can suppose that the linear regions considered are convex. The triangulation is given by
standard barycentric subdivision, defining the barycenter of a
linear region as its mass center.
\end{proof}

\begin{defn}[Subdivision of a measurable simplex]\label{d:subdivision of a measurable simplex}
Let $\bigtriangleup\times S$ be a measurable $n$-simplex. A {\em
measurable subdivision\/} of a measurable simplex is a countable
family of measurable $n$-simplices and MT-embeddings
$\phi_i:\bigtriangleup\times S_i\to\bigtriangleup\times S$ such
that $\phi_i(\bigtriangleup\times S_i)$ determines a usual
subdivision on each fiber; {\em i.e.\/}, the family
$\{\phi_i(\bigtriangleup\times S_i)\cap
(\bigtriangleup\times\{s\})\}$ is a usual subdivision of the simplex
$\bigtriangleup\times\{s\}$ for all $s\in S$.
\end{defn}

\begin{prop}\label{p:measurableadaptedsubdivision}
Let $\UU$ be a measurable countable open covering of a measurable
simplex $\bigtriangleup\times T$. There exists a measurable
subdivision of $\bigtriangleup\times T$ such that, with notation
of Definition~\ref{d:subdivision of a measurable simplex}, each
$\phi_i(\bigtriangleup\times S_i)$ is contained in some element of
$\UU$. In fact, there exists a measurable partition
$\{T'_j\}_{j\in\N}$ of $T$ such that the subdivision on
$\bigtriangleup\times T'_j$ is constant; i.e., the same
subdivision in each fiber $\bigtriangleup\times\{*\}$.
\end{prop}
\begin{proof}
Take a base of open subsets of $\bigtriangleup^n$ given by the
barycentric subdivisions, where the simplices of these
subdivisions are slightly augmented to be open sets and their
closures are also simplices. Order this base of augmented
simplices and denote it by
$\BB=\{\bigtriangleup_1,\dots,\bigtriangleup_i,\dots\}$. By using
Theorem~\ref{t:Srivastava} on each $U\in\UU$, there exists a
sequence $\{T_i\}_{i\in\N}$ of Borel subsets of $T$ such that
$\bigtriangleup\times T=\bigcup_i (\bigtriangleup_i\times T_i)$, and
each $\overline{\bigtriangleup_i}\times T_i$ is contained in some
$U\in\UU$. Observe that, in general, the family $\{T_i\}_{i\in\N}$
is not a partition of $T$.

For each $N\in\N$, let $S^N_{1,...,N}=\bigcap_{i=1}^N T_i$ and let
$I\subset\{1,...,N\}$. Define recursively $S^N_{I}=(\bigcap_{i\in
I} T_{i})\setminus\bigcup_{J\varsupsetneq I}\bigcap_{j\in J}T_j$.
They are a finite number of disjoint measurable transversals, and,
clearly, $\bigcup_{i=1}^{N}\overline{\bigtriangleup_i}\times
T_i=\bigsqcup_{I}(\bigcup_{i\in
I}\overline{\bigtriangleup}_i)\times S^N_I$.

By compactness, each fiber is contained in a finite union
$\bigcup_{i=1}^{N}(\overline{\bigtriangleup_i}\times T_i)$ for some
$N$ large enough; in fact, it is contained in some of the
sets $(\bigcup_{i\in I}\overline{\bigtriangleup}_i)\times S^N_I$. The
family of transversals $S^N_I$ is clearly countable, and let
$\{S_n\}_{n\in\N}$ be an enumeration of this family. Let
$\{S_{n_k}\}_{k\in\N}$ the subfamily of transversals such that
their topological fibers ($\bigcup_{i\in
I(k)}\overline{\bigtriangleup}_i$ for $S_{n_k}=S^{N(k)}_{I(k)}$)
cover $\bigtriangleup$. Let $T'_1=S_{n_1}$, and define recursively
$T'_k=S'_{n_k}\setminus\bigcup_{j=1}^{k-1}S_{n_j}$. This family is
a measurable partition of $T$ and every $T'_k$ is contained in
$S^{N(k)}_{I(k)}$. For each $k\in\N$, Lemmas~\ref{l:attachedunion}
and \ref{l:triangulationlinearregions} provide a triangulation of
$\bigtriangleup$ that induces subdivisions in the family of
simplices $\{\overline{\bigtriangleup}_i\}_{i\in I(k)}$. This
triangulation is extended to a measurable triangulation in the
measurable prism $(\bigcup_{i\in
I(k)}\overline{\bigtriangleup}_{i})\times T'_k$ by taking exactly
the same triangulation on each topological fiber.  Hence the
measurable subdivision induced on $\bigtriangleup\times T$
satisfies the conditions of the statement.
\end{proof}

\begin{cor}\label{cor:invariancesubdivision}
Let $\UU$ be a measurable countable open covering of $\FF$. Then
$i:H^n_{\MT}(\FF,\Gamma)\to H^n_{\UU,\MT}(\FF,\Gamma)$ is an isomorphism for
all $n\in\N$.
\end{cor}
\begin{proof}
Notice that any singular chain $\sg:\bigtriangleup\times T^{\sg}\to \FF$ induces on $\bigtriangleup\times T$ a countable
partition of adapted subdivisions with respect to the covering
$\sg^{-1}\UU$ by using the above proposition. Let $\{T^\sg_i\ |\ i\in\N\}$ denote
the partition corresponding to $T_\sg$, and $\TT^\sg_i$ the subdivision
of each fiber of $\bigtriangleup\times T^\sg_i$. For $t\in T_\sg$,
let $i(t)$ be the unique index such that $t\in T^\sg_i$. The
expression of the inverse map of $i^*$ on closed cochains is
$\rho^*(\omega)(\sg)(t)=\sum_{\bigtriangleup\in\TT^\sg_{i(t)}}\omega(\sg_{|\bigtriangleup\times
T^\sg_{i(t)}})(t)$.
\end{proof}

Let $C^*_{\MT}(U+V)$ be the cochain complex given by the measurable
cochains which vanish on measurable singular simplices that
do not lay in either $U$ or $V$, and let $H_{\MT}^*(U+V)$ denote its cohomology. By the previous corollary,
$H^*_{\MT}(\FF_{U\cup V})$ is isomorphic to $H^*_{\MT}(U+V)$ via the restriction
map. In fact, by using the 5-Lemma, $H^*_{\MT}(\FF,\FF_{U\cup V})$ is isomorphic
to $H^*_{\MT}(\FF,U+V)$.

\begin{cor}
The usual cup product induces a well defined cup product in measurable relative
cohomology:
$$
\smile\;:H^n_{\MT}(\FF,\FF_U,\Gamma)\times H^m_{\MT}(\FF,\FF_V,\Gamma)\to
H^{n+m}_{\MT}(\FF,\FF_{U\cup V},\Gamma)\;.
$$
\end{cor}
\begin{proof}
The cup product gives
$$
\smile\;:H^n_{\MT}(\FF,\FF_U,\Gamma)\times H^m_{\MT}(\FF,\FF_V,\Gamma)\to
H^{n+m}_{\MT}(\FF,U + V,\Gamma)\;.
$$
Now, observe that $H^*_{\MT}(\FF,U+V,\Gamma)$ is isomorphic to
$H^*_{\MT}(\FF,\FF_{U\cup V},\Gamma)$.
\end{proof}

Our goal now is to adapt the long exact sequence of a triple in
order to approach some kind of excision result.

A {\em measurable triple\/} $(X,A,B)$ is a collection of three
MT-spaces such that $A$ is MT-subspace of $X$ and $B$ is
MT-subspace of $A$. Of course, we have
the following short exact sequence of complexes:
$$
0\to C^*_{\MT}(X,A,\Gamma)\to C^*_{\MT}(X,B,\Gamma)\to
C^*_{\MT}(A,B,\Gamma)\to 0\;.
$$
The surjectivity of the map to $C^*_{\text{\rm MT}}(A,B,\Gamma)$ can be proved as follows: giving $\omega\in
C^*_{\MT}(A,B,\Gamma)$, let $\widetilde{\omega}$ be its extension assigning $0$ to simplices that are not contained in $A$; it is clear that this
extension is measurable. Therefore we obtain a short exact
sequence in cohomology:
$$
\cdots\to H^{n}_{\MT}(X,A,\Gamma)\to H^n_{\MT}(X,B,\Gamma)\to
H^n_{\MT}(A,B,\Gamma)\to H^{n+1}_{\MT}(X,A,\Gamma)\to\cdots\;
$$

For the excision statement we refine the conditions on the triple
$(X,A,B)$. Now we require that $\overline{B}$ and $\operatorname{\intr}(A)$ are
measurable sets, and $\overline{B}$ is MT-subspace of $\operatorname{\intr}{A}$.

\begin{rem}
By a result due to Kallman \cite{Kallman}, if a measurable set meets each leaf of a measurable
lamination in a $\sg$-compact set, then its closure and interior are measurable. Under these
conditions the hypothesis on the triple may be reduced to the
usual $A,B\subset X$ and $\overline{B}\subset\intr(A)$.
\end{rem}

\begin{thm}[Excision for measurable laminations]
Let $\UU=\{U,V\}$ be measurable open covering of a measurable
lamination $\FF$. Then $H^*_{\MT}(\FF,U,\Gamma)\cong
H^*_{\MT}(V,U\cap V,\Gamma)$ via the inclusion map $i:(V,U\cap
V)\to (\FF,U)$. Equivalently, if $Z\subset U$ is measurable and
closed, then $H^*_{\MT}(\FF,U,\Gamma)\cong H^*_{\MT}(\FF\setminus
Z, U\setminus Z,\Gamma)$ via the inclusion.
\end{thm}

The equivalence between the two statements is well known
\cite{Hatcher}.

\begin{proof}
We had seen that the inclusion map $C^*_{\MT}(X,U,\Gamma)\to
C^*_{\MT}(U + V,U,\Gamma)$ induces an isomorphism on measurable
cohomology. On the other hand, it is easy to check that the inclusion map $j:C^*_{\MT}(U +
V,U,\Gamma)\to C^*_{\MT}(V,U\cap V)$ also induces an isomorphism in
measurable cohomology; in fact, $j$ is an isomorphism of cochain complexes.
\end{proof}

\begin{rem}
It is easy to check that $H^n_{\MT}(T,\Gamma)=0$ for $n\geq 1$ and
any transversal $T$ of a measurable lamination. Hence,
by exactness, $H^n_{\MT}(\FF,T,\Gamma)\cong H^n_{\MT}(\FF,\Gamma)$
for $n\geq 2$. On the other hand, $H^0_{\MT}(\FF,\Gamma)$ is the group of measurable maps $f:\FF\to\Gamma$ constant on leaves for any MT-space $\FF$.
\end{rem}

\begin{rem}
Observe that the exactness of the Mayer-Vietoris sequence holds for
measurable cohomology. Let $U,V$ be measurable open sets covering
a measurable lamination. The following short sequence is
exact:
$$
0\to C^n_{\MT}(U+V,\Gamma)\to C^n_{\MT}(U,\Gamma)\oplus
C^n_{\MT}(V,\Gamma)\to C^n_{\MT}(U\cap V,\Gamma)\to 0\;.
$$
The usual proof of the Mayer-Vietoris principle can be adapted by checking that measurability is always
preserved.
\end{rem}

Now, we show another description of the singular
measurable cohomology for measurable laminations. It
will become important to define singular
$L^2$~-~cohomology later.

\begin{defn}
Let $T$ be a complete transversal of \FF. An {\em elementary
measurable singular n-simplex\/} relative to $T$ on \FF\ is an
MT-map $\sg:\bigtriangleup^n\times T_\sg \to \FF$ such that
$T_\sg\subset T$ is a Borel subset and
$\sg(\bigtriangleup^n\times\{t\})\subset L_t$ for all $t\in
T_\sg$. Observe that this definition implies that $\sg^{-1}(L)$ is
a countable union of fibers of the measurable simplex for each
leaf $L\in\FF$. The set of elementary $n$-simplices is denoted by
$EC_n(\FF)$.
\end{defn}

Let $\Meas(T,\Gamma)$ be the group of measurable maps
$T\to\Gamma$. An {\em elementary $n$-cochain\/} over the coefficient
ring $\Gamma$ is a map $\omega:EC_n(\FF)\to \Meas(T,\Gamma)$ such
that $\omega(\sg)$ is supported in $T_\sg$ for all $\sg\in
EC_n(\FF)$, $\omega(\sg_{|\bigtriangleup\times
S})=\omega(\sg)\cdot\chi_{S}$ for any measurable $S\subset T_\sg$
and $\omega(\sg(\id\times h))=\omega(\sg)\circ h$ for all
measurable holonomy map $h:A\to B$ between measurable subsets
$A,B\subset T$. The set of elementary $n$-cochains will be denoted
by $EC^n(\FF,\Gamma)$, and it is endowed with a group structure
induced by $\Gamma$.

The coboundary morphism $\delta:EC^{n-1}(\FF,\Gamma)\to
EC^n(\FF,\Gamma)$ is defined by
$$
\delta\omega(\sg)=\sum_{i=0}^{n}I_i\,\omega(\tau_i)\;,
$$
where $\tau_i$ denote the restriction of $\sg$ to each measurable
$(n-1)$-simplex of the boundary and the $I_i$ denote the
orientation factors ($1$ if the orientation of the face agree with the orientation induced by the simplex or $-1$ otherwise). This is well defined since operations in
$\Gamma$ are measurable. Like in the classical setting,
$\delta^2=0$ and we have a cochain complex. The {\em measurable
countable-to-one cohomology\/} groups are defined as usual by
$H^n(\FF,\Gamma)=\ker \delta_n/\Image \delta_{n-1}$.

Let $f:\FF\to\GG$ be an MT-map such that $f^{-1}(L)$ is a countable
union of leaves of $\FF$ for all $L\in\GG$; such a map is said to be {\em countable-to-one\/}. By using Lusin's lemma,
there exist complete transversals, $T$ of $\FF$ and $T'$ of $\GG$, such that $f(T)\subset T'$ and $f:T\to T'$ is
injective. Therefore, measurable countable-to-one cohomology is
functorial with respect to countable-to-one maps and complete transversals satisfying
the above conditions.

\begin{prop}\label{p:elementaryinvariancetransversals}
Elementary cochain complexes relative to different complete
transversals are isomorphic. Hence the cohomology groups are
independent of the choice of the complete transversal.
\end{prop}
\begin{proof}
We use the notation $EC_n(T)$ for the elementary $n$-chains
relative to a complete transversal $T$ and the notation $EC^n(T)$ for the associated elementary $n$-cochains. Obviously, it is sufficient to prove that
$EC^*(T)$ is isomorphic to $EC^*(T')$ for $T\subset T'$. The
inclusion map $i_*:EC_*(T)\hookrightarrow EC_*(T')$ induces a cochain map
$i^*:EC^*(T')\to EC^*(T)$. Now, let us define its inverse. Let
$B_1=T\cap T'$ and $C=T'\setminus T$. By Lusin's lemma
(Lemma~\ref{p:Lusin1}), there exists a countable measurable
partition $\{B_2,B_3,\dots\}$ of $C$ and measurable holonomy maps
$h_i:B_i\to h_i(B_i)\subset T$ for $i>1$. For $t\in T$, let $i(t)$
be the unique positive integer such that $t\in B_{i(t)}$. For
$\sg\in EC_n(T')$ and $i\in\N$, define
$\sg_{i}:\bigtriangleup\times h_{i}(B_{i}\cap T'_\sg)\to\FF$ by
$\sg_{i}(x,t)=\sg(x,h_{i}^{-1}(t))$; observe that $\sg_i\in
EC_n(T)$ for all $i\in\N$. Let $\omega\in C^*(T)$, and define
$\rho\omega(\sg)(t)=\omega(\sg_{i(t)})(h_{i(t)}(t))$, which is an
elementary cochain relative to $T'$. Clearly, $\rho:EC^*(T)\to
EC^*(T')$ is inverse of $i^*$.
\end{proof}
\begin{cor}
Let $\FF$ be a one-leaf foliation. Then the standard singular
cohomology groups are isomorphic to the measurable countable-to-one ones.
\end{cor}

\begin{prop}\label{p:equivsingularlaminated}
Suppose that there exists a countable covering
$\UU=\{U_n\}_{n\in\N}$ such that all finite intersections are
measurably contractible, i.e., there exists a measurable
deformation of the inclusion map to a constant map along the leaves.
Then the measurable singular cohomology groups and the measurable
countable-to-one cohomology ones are isomorphic.
\end{prop}

\begin{proof}
We shall show that this two cohomology groups are isomorphic to a
certain notion of measurable \v{C}eck cohomology groups
related to a nice covering of the measurable lamination,
and therefore, they are isomorphic.

The measurable \v{C}eck cohomology with respect to a measurable
open covering with constant coefficients on a measurable group
$\Gamma$, is defined as follows; of course, it is also possible to
define it by using sheaves. Let
$$
C^s=\prod_{(i_0,...,i_s)}F(U_{i_0}\cap\dots\cap U_{i_s},\Gamma)\,,
$$
where $F(U,\Gamma)$ is the set of measurable maps $f:U\to\Gamma$
constant on the leaves of $\FF_U$. The coboundary map is defined
like in the usual \v{C}eck complex, and the
corresponding cohomology groups are the measurable \v{C}eck cohomology
groups induced by the covering.

Let $S^r(U,\Gamma)$  (respectively, $S^{\prime r}(U,\Gamma)$) denote the set of measurable singular cochains (respectively, countable-to-one)
relative to $U$ with coefficients in $\Gamma$. Consider the following two
double complexes
\begin{align*}
C^{r,s}=\prod_{(i_0,\cdots,i_s)}C^r_{\MT}(U_{i_0}\cap\dots\cap
U_{i_s},\Gamma)\;,\\
C'^{r,s}=\prod_{(i_0,\cdots,i_s)}EC^r(U_{i_0}\cap\dots\cap
U_{i_s},\Gamma)\;,\\
\end{align*}
Observe that $C^{*,-1}=C^*_{\UU,\MT}(\FF,\Gamma)$, and the
coboundary map on $C^{*,n}$ is the usual measurable coboundary
map. In the same way, $C'^{*,-1}=EC^*_{\UU}(\FF,\Gamma)$. For the
vertical rows, the coboundary map is defined like in the
\v{C}eck complex. Of course, $C^{-1,*}=\ker(C^{0,s}\to C^{1,s})$
and $C'^{-1,s}=\ker(EC^{0,s}\to EC^{1,s})$. An easy computation
shows that $C^{-1,*}$ and $C'^{-1,*}$ are the same complex (the
measurable \v{C}eck complex relative to $\UU$).

For $i,j>-1$, the columns and rows of $C$ and $C'$ have trivial cohomology since the rows
are given by measurably contractible spaces and, for the columns, we
have the following obvious cochain homotopy to zero. Suppose that $\UU$ is well
ordered and, for $\sg:\bigtriangleup\to U_{i_0}\cap\dots\cap
U_{i_s}$, define $h\sg=\sg:\bigtriangleup\to U_{i_\sg}\cap
U_{i_0}\cap\dots\cap U_{i_s}$, where $U_{i_\sg}$ is the first
element of $\UU$ containing the image of $\sg$. This definition
also works for degree $-1$, where we consider singular simplices
$\sg:\bigtriangleup\to \FF$ whose image is contained in some
element of $\UU$. The cochain homotopy $h:C^{r,s}\to
C^{r,s-1}$ is defined by $h\omega(\sg)=\omega(h\sg)$. The case of
$C^{\prime r,s}$ is completely analogous.

By a standard argument, the cohomology groups of the first row and
first column are isomorphic. Finally the result follows by
measurable excision.
\end{proof}

\begin{rem}
The hypothesis is not very restrictive and includes many of the
interesting examples. For instance, it holds for
measurable suspensions with a good base, measurable simplicial
spaces (where the leaves are simplicial complexes) or usual
topological foliations.
\end{rem}

\section{Simplicial, $L^r$ and differentiable measurable cohomology}

A measurable lamination may have a measurable simplicial structure. Roughly speaking, it is a simplicial structure on the leaves that varies in a measurable way.
It is natural to adapt to these special cases the concept of
simplicial and cellular cohomology. Also we introduce the $L^2$
measurable cohomology when there exists a transverse invariant
measure and the differentiable measurable cohomology for
differentiable measurable laminations. Original definitions are
given in \cite{Bermudez}.

\begin{defn}[Measurable triangulation
\cite{Bermudez}]\label{d:measurabletriangulation} A measurable
triangulation for a measurable lamination is a measurable
family of triangulations $\{\TT_L\}_{L\in\FF}$. Here,
measurability means that the set of their $n$-simplices are
embedded MT-spaces. The image of barycenters of $n$-simplices,
denoted by $\BB^n$, is a transversal. The function
$\sg^n:\bigtriangleup^n\times\BB^n\to X$, mapping a barycenter to
the embedding $\sg^n_p:\bigtriangleup^n\to L_p$ given by
$\TT_{L_p}$, must be measurable, where $\bigtriangleup^n$ is the
canonical $n$-simplex. A measurable triangulation is of class
$C^m$ if the functions $\sg^n_p$ are $C^m$.
\end{defn}

Let $\TT$ be a
triangulation. An {\em n-cochain\/} over a measurable ring
$\Gamma$ is a measurable map $\omega:\BB^n\to \Gamma$; of course, we identify the
barycenters of $\BB^n$ with the respective $n$-simplex. We denote by $C^n(\TT,\Gamma)$ the set of
simplicial n-cochains; this set is endowed with a ring structure induced by
$\Gamma$. We define the coboundary operator $\delta:C^n(\TT,\Gamma)\to
C^{n+1}(\TT,\Gamma)$ as usual by $\delta\omega:\BB^{n+1}\to
\Gamma$,
$$
\delta\omega(b)=\sum_{\bigtriangleup^n_{p}\subset\partial\bigtriangleup^{n+1}_b}b(p)\,\omega(p)
$$
where $b(p)$ is the orientation of the simplex
$\bigtriangleup^n_p$ induced by $\bigtriangleup^{n+1}_b$.

Clearly, $\delta^2=0$ and we can define the cohomology groups as usual:
$$H^n(\TT,\Gamma)=\Ker \delta_n/\Image \delta_{n-1}\;.$$

\begin{prop}\label{p:equivsimplicialsingular}
Let $(X,\FF)$ be a measurable lamination that admits a
measurable triangulation. Then its measurable singular and simplicial cohomology
groups are isomorphic.
\end{prop}
\begin{proof}
The standard argument used to prove that measurable singular
cohomology is isomorphic to the measurable countable-to-one cohomology
(Proposition~\ref{p:equivsingularlaminated}) can be adapted easily
in order to obtain this result. We only need to check the
existence of a nice covering by measurable open sets such that all
finite intersections of them are also contractible. The covering
will be given by using the star open sets of the $0$-simplices. We claim that there exists a measurable countable partition, $\TT=\{S_n\}_{n\in\N}$, of $\BB^0$ such that the measurable open sets $$U_n=\{\,x\in\FF\ |\ \exists y\in S_n,\, x\in\Star(y)\,\}$$
 form a nice covering satisfying the above conditions, where $\Star(y)$ denotes the open star set around the $0$-simplex $y$; {\em i.e.\/}, the union of the interiors of all simplices containing $y$. Equivalently, we shall prove that any pair of points of each $S_n$ are not connected by any $1$-simplex.

Let $\sg^1:[0,1]\times\BB^1\to\FF$ be the $1$-simplicial structure of our measurable triangulation. In the measurable prism $[0,1]\times \BB^1$, we have a canonical simplicial structure where the set of $0$-simplices is $C^0=\{0,1\}\times \BB^1$ and the restriction map $\sg^1:C^0\to\BB^0$ has countable fibers. By Lusin's lemma (Proposition~\ref{p:Lusin1}), there exists a measurable partition $\{T_n\}_{n\in\N}$ of $C^0$ such that $\sg^1:T_n\to \BB^0$ is a measurable injection for all $n\in\N$. Each $T_n$ induces a measurable holonomy map: endow $\{0,1\}$ with the group structure of $\Z_2$ and define $h_n:\sg^1(T_n)\to \BB^0$ by $h_n(\sg^1(x,t))=\sg^1(x+1,t)$.

Let $\BB^0_m$ be the measurable set of $0$-simplices that meet exactly $m$ edges of $\BB^1$. This family gives a countable partition of $\BB^0$. Each non-empty intersection $K^m_{i_1,\dots,i_m}=\sg^1(T_{i_1})\cap\dots \sg^1(T_{i_m})\cap \BB^0_m$ defines a domain where only the holonomy maps $h_{i_1},...,h_{i_m}$ are defined. Observe that the sets $K^m_{i_1,\dots,i_m}$, $m\in\N$, form a measurable countable partition of $\BB^0$. If $K^m_{i_1,\dots,i_m}\cap h_{i_j}(K^m_{i_1,\dots,i_m})=\emptyset$ for $1\leq j\leq m$, then this transversal satisfies our conditions, otherwise we claim that there exists a measurable countable (in fact finite) partition of $K^m_{i_1,\dots,i_m}$ where this is true in each element of the partition. Endow $T_m=\BB^0_m\cap U$ with a Polish topology isomorphic to $[0,1]$ (the other cases are trivial). For each $x\in K^m_{i_1,\dots,i_m}$, consider a family of open neighborhoods $V^x_{n}$ such that $\bigcap_{n} V^x_n=\{x\}$. Hence for $n$ large enough, $V^x_n\cap h_{i_j}(V^x_n)=\emptyset$ for $1\leq j\leq m$; otherwise there exists an edge connecting $x$ with itself, which contradicts the definition of triangulation. Notice that $h_{i_j}(V^x_n)$ is not open in general, but it is measurable. Now, by compactness, we can choose neighborhoods $V_1,...,V_N$ covering $K^m_{i_1,\dots,i_m}$ such that any pair of points on each one of them are not connected by any $1$-simplex. Then the desired partition of $K_{i_1,\dots,i_m}^m$ is inductively defined by $B^{i_1,\dots,i_m}_1=V_1$ and $B^{i_1,\dots,i_m}_k=V_k\setminus(\bigcup_{i=1}^{k-1} V_i)$.

Finally, the partition $\TT$ is given by all the previous partitions $\{B^{i_1,\dots,i_m}_k\}^{m\in\N}_{k\in\N}$.
\end{proof}
\begin{cor}
The measurable simplical cohomology does not depend on the measurable triangulation.
\end{cor}

In the setting of measurable simplicial cohomology we can restrict
to a fixed complete transversal parametrizing the measurable
simplices. Suppose that there exists a transverse invariant
measure $\LB$ and $\Gamma=\R$ or $\C$ (or a measurable subgroup of
them). Then, we can work with $L^n$-measurable cochains, which are
equivalence classes of $L^n$-maps $f:\BB^n\to \Gamma$ that are $l^n$ on
each leaf, with the equivalence relation defined by being \LB-almost everywhere equal.
The induced cohomology groups are the $L^n$-measurable cohomology
groups associated to the measurable triangulation. For a good
definition of the coboundary map, assume that the measurable
triangulation is regular; {\em i.e.\/},  there is a uniform bound for the number of simplices
shearing any face. The invariance
of the measure implies that these cohomology groups do not depend
on the measurable triangulation. These $L^n$ cohomology groups are not isomorphic
in general to the measurable cohomology groups, but they
give a simplification of the description of the measurable
cohomology groups and help in the search of nonzero
cocycles. In the $L^2$-case, we can give a structure of Hilbert
space to the $L^2_\LB$-measurable cochains, and define the
$\LB$-Betti numbers like the Murray-von Neumann dimension of the
space of harmonic cochains \cite{Connes,Bermudez}, which is
isomorphic to the reduced $L^2$-cohomology, defined as the
quotient of the kernel of each $\delta_n$ over the closure of the image
of $\delta_{n-1}$. For measurable laminations with differentiable
structure on the leaves, we can define the differentiable
measurable cohomology groups \cite{Bermudez,Heitsch-Lazarov}. Of
course, measurable differentiable classes are measurable sections of $\bigwedge T\mathcal{F}^*$ smooth on the leaves, and the coboundary operator is the exterior derivative on
the leaves; the measurability is preserved by the exterior derivative since partial derivatives are computed by a limit of measurable maps. These cohomologies are
related to the leafwise cohomology.

Finally, observe that it is also possible to define a version of the cellular cohomology if we consider a CW-structure on the leaves varying measurably (similarly to a measurable triangulation).

\section{Homotopy invariance of the $L^2$-cohomology}
One of our motivations to define the measurable singular
cohomology is its homotopy invariance, which is a simple way to
prove that other isomorphic cohomologies are homotopy invariants.
The $L^2$-cohomology groups are defined with respect to a measurable
triangulation. Hence, their homotopy invariance is a good
problem. In \cite{Heitsch-Lazarov}, the proof of this fact is based on the
notion of simplicial approximation, without introducing the
concept of singular $L^2$-cohomology. We solve it by defining a
singular version of these groups where the homotopy invariance is easy to
prove, and then we show that usual $L^2$-cohomology and singular
$L^2$-cohomology are isomorphic. Of course, we need some
conditions on the ambient space of the lamination. We assume the
existence of a finite regular foliated atlas such that the transversal
associated to each chart has finite $\LB$-measure,
where $\LB$ is the transverse invariant measure. Also, it is
supposed that there is a riemannian metric on the leaves
that varies measurably on the ambient space.
\begin{defn}
Let $T$ be a complete transversal to $\FF$ with finite
$\LB$-measure. A {\em singular $L^2$-cochain\/} is the equivalence class
of an elementary cochain $\omega$ such that $\int_{T_{\sg}}
|\omega(\sg(-,t))|^2 d\LB(t)<\infty$ for all elementary simplex
$\sg:\bigtriangleup^n\times T_\sg\to\FF$ relative to $T$. The equivalence relation
is given by the ``$\LB$-almost everywhere equal'' equivalence relation.
\end{defn}

\begin{rem}
Notice that $L^2$-cochains do not form a Hilbert space. There is no obvious direct way to endow this space with
a scalar product.
\end{rem}

Of course, the operator $\delta$ preserves singular $L^2$-cochains
and we can define the singular $L^2$-cohomology groups $L^2_{\LB}
H^n_{\MT}(\FF,\Gamma)$. Let $f:(\FF,\LB)\to (\GG,\Delta)$ be a
countable-to-one MT-map such that $\LB(T),\Delta(T')<\infty$ for
complete transversals $T$ and $T'$ so that $f:T\to T'$ is injective. This kind of map is called {\em comparable \/} and induces a homomorphism in singular
$L^2$-cohomology, $f^*:L^2_\Delta H^n_{\MT}(\GG,\Gamma)\to L^2_\LB
H^n_{\MT}(\FF,\Gamma)$.

\begin{prop}
The singular $L^2$-cohomology groups are independent of the choice of
the complete transversal.
\end{prop}
\begin{proof}
It can be proved like the analogous statement for countable-to-one
cohomology (Propostion~\ref{p:elementaryinvariancetransversals}).
\end{proof}

\begin{prop}[MT-homotopy invariance]\label{p:L2invariancetransversals}
Let $f,g:(X,\FF,\LB)\to (Y,\GG,\Delta)$ be MT-homotopic comparable countable-to-one maps; i.e., there exist a countable-to-one MT-map $H:(X\times\R,\FF\times\R,\LB)\to (\GG,\Delta)$ such that $H(x,0)=f(x)$ and $H(x,1)=g(x)$, where $\FF\times\R$ denotes the measurable lamination whose leaves are $L\times\R$, $L\in \FF$; observe that the homotopy must be also comparable by Lusin's lemma. Then $f^*$ and $g^*$ induce the same homomorphism in singular
$L^2_\LB$-cohomology.
\end{prop}
\begin{proof}
Looking at the proof of the
Proposition~\ref{p:invariancetransversals}, the fact that the
homotopy preserves transversals of finite measure means that the
induced cochain homotopy preserves measurable singular
$L^2$-cochains. Also, observe that the cochain homotopy, $P$, is
well defined at the level of elementary cochains and the homotopy
condition, $g^* - f^*=\partial P + P \partial$, holds.
\end{proof}

\begin{prop}
The singular $L^2$-cohomology groups are isomorphic to the
$L^2$~-~simplicial ones.
\end{prop}
\begin{proof}
It is the same as the proof of
Proposition~\ref{p:equivsimplicialsingular} with the obvious
changes. We only need to give the correct notion of singular
$L^2$-\v{C}eck cohomology. Remember that the chosen nice
covering is the covering $\{U_n\}_{n\in\N}$ formed by star open sets around a suitable partition of the
$0$-simplices (see the proof of Proposition~\ref{p:equivsimplicialsingular}). The $L^2$-\v{C}eck cohomology is defined like the measurable one by taking $L^2$-maps $f:U_n\to\Gamma$ constant along the star open sets.
\end{proof}

\begin{quest}
The question about the homotopy invariance of the $\LB$-Betti
numbers was formulated by Connes in \cite{Connes} and solved by
Heitsch and Lazarov in \cite{Heitsch-Lazarov}. Could the singular
$L^2$-cohomology provide a simpler proof of this fact?
\end{quest}

\section{Computation on examples and applications}

\begin{exmp}[Product and trivial Polish laminations \cite{BermudezTesis}]
Let $C$ be a simplicial complex such that each simplex meets only finitely many other simplices. Then the measurable simplicial cohomology of the product MT-space $C\times T$, where $T$ is a standard Borel space, can be identified to the space of measurable maps $f:T\to H^*(C,\Gamma)$, where $H^*_{\bigtriangleup}(C,\Gamma)$ denotes the usual cohomology of $C$. A similar result holds for the $L^r$-cohomology when we consider a measure $\LB$ on $T$, obtaining $L^r_\LB H^*_{\MT}(C\times T,\Gamma)\cong L^r(T,\Gamma;\LB)\otimes H^*(C,\Gamma)$.
\end{exmp}

\begin{exmp}[Wedges]
We compute here the measurable cohomology groups of a measurable
wedge. Let $\FF,\GG$ be measurable laminations, let $T$ be a complete
transversal of $\FF$ and $\gamma:T\to\GG$ a measurable injection
such that $\gamma(T)$ is a complete transversal of $\GG$. Let
$\pi:\FF\sqcup\GG\to \FF\vee\GG$ be the projection onto the
wedge construction identifying each $t\in T$ with $\gamma(t)$. Suppose
that $T$ consists of isolated points, and there are measurable atlases of $\mathcal{F}$ and $\mathcal{G}$ with contractible plaques. Then there exists a
measurable open set $U\subset \FF\vee \GG$ and an MT-homotopy $H:U\times [0,1]\to \FF\vee \GG$ such that
$\pi(T)\subset U$, $H(\cdot,0)$ is the identity map on $\pi(T)$, and $H(\cdot,1)$ is a retraction of $U$ to $\pi(T)$. Moreover
$\pi^{-1}(U)=U'\sqcup U''$, where $T\subset U'$,
$\gamma(T)\subset U''$, and $H$ defines MT-homotopies, $H':U'\times[0,1]\to\mathcal{F}$ and $H'':U''\times[0,1]\to\mathcal{G}$, such that $H'(\cdot,0)$ and $H''(\cdot,0)$ are identity maps, and $H'(\cdot,1)$ and $H''(\cdot,1)$ are retractions to $T$ and $\gamma(T)$, respectively. By
homotopy invariance,
$$C^n(\FF,T,\Gamma),\quad C^n(\GG,\gamma(T),\Gamma),\quad C^n(\FF\vee\GG,\pi(T),\Gamma)$$
are respectively isomorphic to
$$C^n(\FF,U',\Gamma),\quad C^n(\GG,U'',\Gamma),\quad C^n(\FF\vee\GG,U,\Gamma)\;.$$
Consider the measurable open coverings $\UU=\{\FF\setminus
T,\GG\setminus \gamma(T), U\}$, $\UU'=\{\FF\setminus T,U'\}$,
$\UU''=\{\GG\setminus \gamma(T), U''\}$ of $\FF\vee\GG$, $\FF$
and $\GG$, respectively. It is clear that the
complexes $C^*_{\UU}(\FF\vee\GG,U,\Gamma)$ and
$C^*_{\UU'}(\FF,U',\Gamma)\oplus C^*_{\UU''}(\GG,U'',\Gamma)$
are isomorphic. Therefore, by Corollary
\ref{cor:invariancesubdivision}, we obtain that
$$H^*(\FF\vee\GG,\pi(T),\Gamma)\cong H^*(\FF, T,\Gamma)\oplus
H^*(\GG,\gamma(T),\Gamma)\;.$$
\end{exmp}

\begin{exmp}
Let $(T^2,\FF_\alpha)$ be the K\"onecker flow, given as a
suspension of the rotation $R_{\alpha}:S^1\to S^1$ of $2\pi\alpha$
radians. The case where $\alpha$ is rational is trivial. Thus suppose that
$\alpha$ is irrational and let us prove that
$H^{1}_{\MT}(\FF_\alpha,\Z_2)$ is non-trivial. The projection of
$[0,1]\times S^1$ to $T^2$, given by the suspension, induces a
measurable triangulation of $\FF_\alpha$, where the $0$-skeleton is
the projection of $\{0\}\times S^1$ and the $1$-skeleton is the
projection of $[0,1]\times S^1$ (we consider the set of
barycenters as the projection of $\{1/2\}\times S^1$). Of course, the
$0$ and $1$ measurable cochains are measurable maps $f:S^1\to
\Z_2$. We show that the $1$-cochain, $\omega\equiv 1:S^1\to \Z_2$,
is non-trivial. If $\omega$ is trivial, then there exits
a measurable map $f:S^1\to \Z_2$ and $1=\omega=f\circ R_{\alpha} - f$. Of
course, it is also true that $1=\omega\circ R_{\alpha}=f\circ
R_{2\alpha} - f\circ R_{\alpha}$. Hence $0=f\circ R_{2\alpha}-f$,
showing that $f$ is $2\alpha$ invariant. By ergodic arguments, $f$ is
constant almost everywhere, and therefore $f\circ R_\alpha - f=0$
almost everywhere, which is a contradiction.

In higher dimension, let $\FF_{\alpha_1,\dots,\alpha_n}$
the foliation of $T^{n+1}$ given by the suspension of minimal
rotations of $S^1$ with angles $2\pi\alpha_1,\dots,2\pi\alpha_n$ radians, where $\alpha_1,\dots,\alpha_n$ are $\Q$-linear independent. Each
leaf of this foliation is an hyperplane dense in $T^{n+1}$.

Let $[0,1]^n\subset\R^n$ be the unit cube and let $T = S^1$.
The product $[0,1]^n\times T$ defines a measurable CW-structure on
$\FF$ given by the suspension projection $p:[0,1]^n\times
T\to\FF$. Let $\omega\equiv 1:T\to\Z_2$, which is
a CW-cochain of dimension $n$. Let
$R_{\alpha_i}$ be the rotation of $S^1$ by $2\pi\alpha_i$. If
there exists a CW-cochain $\theta$ such that $\delta\theta=\omega$,
then
\begin{equation}
\sum_i(\theta_i\circ R_{\alpha_i}+\theta_i)=1 \label{eq:I}\;,
\end{equation}
where
$\theta_i=\theta(p_{|e_i\times T})$ and $e_i=\{(x_1,\dots,x_n)\in [0,1]^n\ |\ x_i=0\}$. The proof that $\omega$ is non-trivial is more
difficult than in the one dimensional case. First of all, any
measurable map $f:S^1\to\Z_2$ can be considered as a characteristic
map $\chi_B:S^1\to \Z_2$, where $B$ is a measurable subset on $S^1$.
Given measurable subsets $B,C\subset S^1$, it is clear that $\chi_B +
\chi_C=\chi_{B\bigtriangleup C}$, where $B\bigtriangleup
C=(B\setminus C) \cup (C\setminus B)$. Of course, $\chi_B\circ
R_\alpha=\chi_{R_{-\alpha}B}$, hence $\chi_B\circ R_\alpha +
\chi_B=\chi_{(R_{-\alpha}B)\bigtriangleup B}$. We want to show that,
for any family $\{B_1,\dots,B_n\}$ of $n$ measurable subsets of $S^1$,  such
that $\sum_i \chi_B\circ R_\alpha + \chi_B=1$, the set
$$
Z=\left\{\,z\in
S^1\ |\ \sum_i \chi_{B_i}\circ R_\alpha + \chi_{B_i}=0\,\right\}
$$
has positive measure, which obviously contradicts \eqref{eq:I}. It is not difficult to see this in the case where each
$B_i$ is formed by a finite disjoint union of open arcs, and, in
fact, the measure of $Z$ depends on the lengths of these arcs. Of
course, the contradiction comes from the existence of integers
$m,m_1,...,m_n$ such that $2\pi m_1\alpha_1+\dots +
2\pi m_n\alpha_n=2\pi m$.

Now, let $\{B_1,...,B_n\}$ be a family of measurable sets and let
$\{U^i_n\}_{n\in\N}$ be a family of sequences of open sets in
$S^1$, where each $U^i_n$ is a finite disjoint union of open arcs such
that $B_i\subset U^i_n$ for all $n$ and $\length(U_i\setminus
B_i)<2^{-n}$. By induction on the dimension, we can suppose that each $B_i$
has positive measure. Hence the lengths of the arcs $U^i_n$
converge to a positive number. Let
$$
Z_n=\left\{\,z\in S^1\ |\ \sum_i
\chi_{U^i_n}\circ R_\alpha + \chi_{U^i_n}=0\,\right\}\;.
$$
By the previous
observation, $\liminf_n \length(Z_n)>0$, and it is easy to see that this
limit equals $\length(Z)$.
\end{exmp}

\begin{ack}
This paper contains part of my PhD thesis, whose advisor is Prof.
Jes\'{u}s A. \'{A}lvarez L\'{o}pez. I also want to thank Prof.
Elmar Vogt for his valuable advice and help in the development of
this work.
\end{ack}

\end{document}